\documentclass{amsart}
\numberwithin{equation}{section}
\usepackage{amssymb}
\usepackage[all]{xy}
\let\cal\mathcal

\def\Cscr{{\cal C}}
\def\Dscr{{\cal D}}

\let\blb\mathbb
\def\CC{{\blb C}}

\def \AA{{\blb A}}
\def \ZZ{{\blb Z}}

\def \HH{{\blb H}}

\def\Ann{\operatorname{Ann}}

\def\Mod{\operatorname{Mod}}

\def\length{\mathop{\text{length}}}

\def\Supp{\mathop{\text{\upshape Supp}}}

\def\rad{\operatorname {rad}}

\def\Spec{\operatorname {Spec}}

\def\Ext{\operatorname {Ext}}
\def\Hom{\operatorname {Hom}}

\def\RHom{\operatorname {RHom}}

\def\Tr{\operatorname {Tr}}
\def\coker{\operatorname {coker}}

\def\r{\rightarrow}

\newtheorem{lemma}{Lemma}[section]
\newtheorem{proposition}[lemma]{Proposition}
\newtheorem{theorem}[lemma]{Theorem}

\theoremstyle{definition}

\newtheorem{conjecture}[lemma]{Conjecture}

{

}

\theoremstyle{remark}

\newtheorem{remark}[lemma]{Remark}

\newdimen\uboxsep \uboxsep=1ex
\def\uboxn#1{\vtop to 0pt{\hrule height 0pt depth 0pt\vskip\uboxsep
\hbox to 0pt{\hss #1\hss}\vss}}

\def\uboxs#1{\vbox to 0pt{\vss\hbox to 0pt{\hss #1\hss}
\vskip\uboxsep\hrule height 0pt depth 0pt}}

\def\HH{\operatorname{HH}}
\def\CC{\operatorname{C}}
\def\Ob{\operatorname{Ob}}

\def\cone{\operatorname{cone}}

\title{On involutivity of $p$-support}
\author{Michel Van den Bergh}
\email{michel.vandenbergh@uhasselt.be}
\address{Universiteit Hasselt\\ Universitaire Campus\\ 3590 Diepenbeek}
\thanks{The author is a senior researcher at the FWO}
\keywords{$p$-support, involutivity, Gabber's theorem}
\subjclass{16S32}
\begin{document}
\begin{abstract} The $p$-support of a holonomic $\Dscr$-module was introduced by Kontsevich.
Thomas Bitoun in his PhD thesis proved several properties of $p$-support
 conjectured by Kontsevich. In this note we give an alternative proof
for involutivity by reducing it to a slight
extension of Gabber's theorem  on the integrability
of the characteristic variety. For the benefit of the
reader we review how this extension follows from Kaledin's proof
of Gabber's theorem.
\end{abstract}

\maketitle
\section{Introduction}
Let $X$ be a smooth affine variety over a field $K$ and let $N$ be
a holonomic $\Dscr$-module.  The \emph{singular support} of $N$
is a Langrangian conical subvariety of $T^\ast(X)$ whose construction
is standard.

\medskip

For a specialization
 $(X_k,N_k)$ of $(X,N)$ to a field of finite characteristic~$k$ Kontsevich defines
the \emph{$p$-support} of $N$ as the support of $N_k$ considered as coherent
sheaf on $T^\ast(X_k)^{(1)}$ (using the theory of crystalline differential
operators in characteristic $p$, see \cite{BMR}). He
conjectured that if the specialization $K\r k$ is sufficiently generic
then the $p$-support is Lagrangian. Since $p$-support is usually not conical
it is potentially a
 finer invariant than singular support.

Kontsevich's conjecture
was proven by Thomas Bitoun in his PhD thesis \cite{Bitoun}.
In particular involutivity
required subtle geometric arguments based on Hodge theory. In this note
we give an algebraic proof of involutivity by reducing it to a slight
extension of Gabber's celebrated theorem on the integrability
of the characteristic variety \cite{Gabber,Kaledin1}. To this end
we use the observation by Belov-Kanel and Kontsevich that the Poisson
bracket on $T^\ast(\AA^n_{\ZZ/p\ZZ})^{(1)}$ is encoded in the lifting of
differential
operators on $\AA^n_{\ZZ/p\ZZ}$ to differential operators on $\AA^n_{\ZZ/p^2\ZZ}$.

\medskip

Gabber's proof is elementary and very general but also quite intricate. On the other hand a conceptual proof
of Gabber's theorem has been given by Kaledin in~\cite{Kaledin1}. After some
adaptations it applies to our setting as well. For the benefit of the reader this is
explained in Appendix~\ref{ref-B-10}.

\section{Acknowledgement} 
The work on this paper was done while the author was visiting the IHES
during the autumn of 2010. He hereby thanks the IHES for its very
pleasant working conditions. He also thanks Thomas Bitoun,
Maxim Kontsevich and Gon\c calo Tabuada for useful discussions.
\section{The $p$-support of a $\Dscr$-module}
We recall some notions introduced by Kontsevich in \cite{Ko11}.
Let $X$ be a smooth affine variety over a field $K$ and let
$\Dscr(X)$ be the ring of differential operators on~$X$. 

Let ${N}$ be a finitely generated $\Dscr(X)$-module.  Then we may find a
subring $R\subset K$ of finite type over $\ZZ$ such that $X$ has
a smooth model $X_R/\Spec R$ and ${N}$ is obtained by base extension
from a finitely generated $\Dscr(X_R)$-module ${N}_R$, where
$\Dscr(X_R)$ is the ring of so-called $R$-linear ``crystalline'' differential
operators \cite{BMR}.  I.e. the differential operators that are linear
combinations of compositions of $R$-linear derivations.

For any prime $p$ and any maximal ideal $v$ in $R/pR$ with residue
field $k_{p,v}$ one obtains
a reduced scheme 
\def\Supp{\operatorname{Supp}\nolimits}
\[
\Supp\nolimits_{p,v}({N}_R)\overset{\text{def}}{=}\Supp({N}_R\otimes_{R} k_{p,v})\subset T^\ast(X_{k_{p,v}})^{(1)}
\]
where we have used:
\begin{theorem}\cite{BMR}
$\Dscr(X_{k_{p,v}})$ is an Azumaya algebra of rank $p^{2\dim X}$ over its
center which is given by the global sections of  $T^\ast(X_{k_{p,v}})^{(1)}$.
\end{theorem} 
 Note that  $T^\ast(X_{k_{p,v}})^{(1)}$ carries
a canonical symplectic form inherited from $T^\ast(X_{k_{p,v}})$ by Frobenius
pullback.
In \cite{Ko11} Kontsevich made the following conjecture
\begin{conjecture}
\label{ref-3.2-0}
If ${N}$ is holonomic then there exists a dense open subset $U$ in $\Spec R$
such  that for all $(p,v)\in U$, $\Supp\nolimits_{p,v}({N}_R)$ is Lagrangian. 
\end{conjecture}
This conjecture was proved in \cite{Bitoun} using subtle geometric 
arguments based on Hodge theory.

In the current note we give an algebraic proof of the involutivity part of Conjecture \ref{ref-3.2-0} 
by reducing it to a slight
extension of Gabber's theorem  on the integrability
of the characteristic variety \cite{Gabber,Kaledin1}.

We first use the observation made in \cite[\S4]{Bitoun} that it is 
sufficient to understand the case $X_R=\AA^n_R$.\footnote{To carry out 
the reduction to $\AA^n$, Bitoun uses a closed embedding $X\subset \AA^n$
combined with the formalism of $\Dscr$-modules in finite characterictic. An
alternative is to use etale local coordinates on $X$. This amounts
to constructing an open affine covering $X=\bigcup_i U_i$ together with etale
maps $U_i\r \AA^{\dim X}$.} 
Subsequently we use the following result by Belov-Kanel and Kontsevich:
\begin{lemma} \label{ref-3.3-1} \cite[Lemma 2]{BKK} Assume $X_R=\AA^n_R$.
Let $k'_{p,v}$ be a $\ZZ/p^2\ZZ$-flat quotient of $R/p^2R$ lifting of $k_{p,v}$. Then $\Dscr(X_{k'_{p,v}})$ is a
first order deformation of $\Dscr(X_{k_{p,v}})$ (see \S\ref{ref-4-2} below)
and the corresponding bracket (see below)  on $ T^\ast(X_{k_{p,v}})^{(1)}$
coincides with the natural Poisson bracket on $T^\ast(X_{k_{p,v}})^{(1)}$ used above,
up to sign.
\end{lemma}
\begin{proof} $\Dscr(X_{k'_{p,v}})$ and $\Dscr(X_{k_{p,v}})$ are Weyl algebas
over $k'_{p,v}$ and $k_{p,v}$ respectively. 
One may then invoke \cite[Lemma 2]{BKK}.
\end{proof}
This lemma will be combined with the following result by Thomas Bitoun
\begin{theorem} \cite[Theorem 5.3.2]{Bitoun}
Assume that $X_R=\AA^n_R$ and
that ${N}$ is holonomic.
There exists a dense open subset $U\subset \Spec R$
such that for all $(p,v)\in U$  the length of ${N}$ at the
generic points of $\Supp_{p,v}({N}_R)$ is bounded by $e({N}) p^{\dim X}$ 
where $e({N})$ is a constant depending only on ${N}$. 
\end{theorem}
\begin{proof}[Proof of Conjecture \ref{ref-3.2-0}]
  As before we assume $X_R=\AA^n_R$ and ${N}$ holonomic. Bitoun has
  proved \cite[Theorem 3.1.1]{Bitoun} that $\Supp_{p,v}({N}_R)$ is equidimensional of dimension $\dim X$
for $(p,v)$ in a dense open subset of $\Spec R$. Hence it remains
to show that $\Supp_{p,v}({N}_R)$ is coisotropic.

This follows from Lemma \ref{ref-3.3-1} and
Proposition \ref{ref-4.2-3} below with $A=\Dscr(X_{k_{p,v}})$,
$A'=\Dscr(X_{k'_{p,v}})$ (see Lemma \ref{ref-3.3-1}), $M={N}_{k_{p,v}}$, $M'={N}_{k'_{p,v}}$,
$k=k_{p,v}$, $n=\dim X$
and $p>e(M)$.
\end{proof}

\section{Involutivity of $p$-support}
\label{ref-4-2}
Below $k$ is a field of characteristic $p$.
Assume that $A$ is a $k$-algebra and $A'$ is a first order deformation of $A$
in the sense that
$A'$ contains a central element $h$ such that $A'/hA'=A$ and $\Ann_{A'}(h)=hA'$.
Then the center $Z=Z(A)$ of $A$  carries an anti-symmetric biderivation (a ``bracket'') given by
\[
\tilde a\tilde b-\tilde b\tilde a=h\{a,b\}
\]
where $a,b\in Z$ and $\tilde{a}$, $\tilde{b}\in A'$ are arbitrary liftings
of $a,b$.

\begin{lemma}
Let $n>0$ and put $B'=M_n(A')$, $B=M_n(A)$.
 Then $B'$ is a first order deformation of $B$
and $Z(B)\cong Z$. The bracket induced on $Z(B)$ is the same as the
one on $Z$. 
\end{lemma}
\begin{proof} The center of $B$ consists of diagonal matrices with
central, identical entries. These may be lifted to diagonal
matrices with identical entries. From this the lemma easily follows. 
\end{proof}
\begin{proposition}
\label{ref-4.2-3}
  Assume in addition that $A$ is an Azumaya algebra over its center $Z$ and $Z/k$ is
  finitely generated and regular. Assume that $A$ has constant rank
  $r^2$ over its center with $r=p^n$. Assume that $M$ is a finitely
  generated left $A$-module and $M'$ is a first order $A'$-defor\-mation
  of $M$ in the sense that $M'/hM'=M$ and $\Ann_{M'}(h)=hM'$. Let
  $I=\rad \Ann_{Z} M$.

Assume that the multiplicities  ($=$ length) of $M$ over the different primary components
of $I$ are of the form $bp^n$ with $(b,p)=1$. Then $I$ is coisotropic. That
is 
\[
\{I,I\}\subset I
\]
\end{proposition}
\begin{proof} 
  The set of all lifts of the regular elements in $Z$ forms an Ore set in
  $A'$. Hence we may regard $A'$ as a sheaf on $\Spec Z$.  It suffices to
  prove the statement in the generic points of $V(I)$. Hence we may
  assume that $Z$ is local with maximal ideal~$m$ and $M$ has finite
length equal to $bp^n$ as in the statement of the proposition. 

From the conditions
  on the rank of $M$ we find that $A/mA$ is split. Without loss of
  generality we may replace $Z$ by its completion and then $A$ itself
  is split.  Thus $A=M_r(Z)$. By lifting idempotents we find
$A'=M_r(Z')$ where $Z'$ is a (normally non-commutative) 
first order deformation of $Z$ (note that we may assume $h\in Z'$)  and by the above lemma it induces the same
biderivation on $Z$ as~$A'$. 

We also have $M'=M_{r\times 1}(M_0')$, $M=M_{r\times 1}(M_0)$ where 
$M'_0$ is a first order deformation of $M_0$  and the latter is a
$Z$-module of finite length, equal to $b$.

Thus the length of $M_0$ is not divisible by $p$. To conclude we note
that Gabber's proof in \cite{Gabber} yields that $m$ is involutive
under these conditions. See \cite[p468, first display]{Gabber}.

 After some
adaptations Kaledin's proof of Gabber's theorem applies to our setting as well. This is
explained in Appendix~\ref{ref-B-10}. 
\end{proof}

\appendix
\section{Hochschild homology and the trace map}
\label{ref-A-4}  This
is a preparatory section for the next one in which we discuss
Kaledin's proof of Gabber's theorem.
In this section we discuss some properties of the trace map.
These
properties are well-known to experts but for
lack of a succinct reference we give proofs for some its properties here using
results in \cite{kellerexact}.

Let $K$ be a field.
If $\mathfrak{a}$ is a $K$-linear DG-category then we may define its
Hochschild homology $\HH_\ast(\mathfrak{a})$ by the usual standard
complex which we denote by $\CC_\ast(\mathfrak{a})$. It will be
convenient to think of $\CC_\ast(\mathfrak{a})$ as the $\oplus$-total complex of a bicomplex
$\CC_{\ast}^{\ast}(\mathfrak{a})$ such that
\[
\CC_{n}^m(\mathfrak{a})=\bigoplus_{A_0,\ldots,A_{n}\in \Ob(\mathfrak{a}),\sum_i m_i=m} 
\mathfrak{a}(A_{n-1},A_n)_{m_{n-1}}\otimes_K\cdots \otimes_K \mathfrak{a}(A_{0},A_1)_{m_0}
\]
with the standard differentials $d_{\operatorname{Hoch}}:\CC_{n}^m(\mathfrak{a})\r \CC_{n-1}^m(\mathfrak{a})$, $d_{\mathfrak{a}}:\CC_{n}^m(\mathfrak{a})\r\CC_{n}^{m+1}(\mathfrak{a}) $.

 Let $A\in\Ob( \mathfrak{a})$.  The inclusion $\mathfrak{a}(A,A)\r \CC^{\ast}_0(\mathfrak{a})$ defines a map of complexes
\[
\mathfrak{a}(A,A)\r \CC_\ast(\mathfrak{a})
\]
which on the level of cohomology defines the so-called \emph{trace} map
\[
\Tr_A:H^\ast(\mathfrak{a}(A,A))\r \HH_{-\ast}(\mathfrak{a})
\]
Let $\Mod(\mathfrak{a})$ be the category of right $\mathfrak{a}$-modules\footnote{Right $\mathfrak{a}$-modules
are contravariant functors $M:\mathfrak{a}\r \operatorname{Ab}$. In other words a collection of
abelian groups $M(A)$, $A\in\mathfrak{a}$ which depends contravariantly on $A$.}. 
We will identify $\mathfrak{a}$ with a full subcategory of $\Mod(\mathfrak{a})$
through the Yoneda embedding. Let $\tilde{\mathfrak{a}}$ be the category
of perfect DG-modules in $\Mod(\mathfrak{a})$. By \cite[Thm 2.4b)]{kellerexact}
the functor $\mathfrak{a}\r \tilde{\mathfrak{a}}$ induces an isomorphism $\HH_\ast({\mathfrak{a}})\cong\HH_\ast(\tilde{\mathfrak{a}})$.  
For $A\in\Ob(\tilde{\mathfrak{a}})$ we define the corresponding trace map as the composition
\[
\Tr_A:\Ext^\ast_{\mathfrak{a}}(A,A)=H^\ast(\tilde{\mathfrak{a}}(A,A))\r \HH_{-\ast}(\tilde{\mathfrak{a}})\cong
 \HH_{-\ast}(\mathfrak{a})
\]
The following properties follow directly from the definition
\begin{lemma} (Functoriality) If $F:\mathfrak{a}\r \mathfrak{b}$ is
a DG-functor then the following diagram is commutative
\[
\xymatrix{
\Ext^\ast_{\mathfrak{a}}(A,A)\ar[r]^{\Tr_A}\ar[d] & \HH_{-\ast}(\mathfrak{a})\ar[d]\\
\Ext^\ast_{\mathfrak{b}}(FA,FB)\ar[r]_{\Tr_{FA}} & \HH_{-\ast}(\mathfrak{b})
}
\]
\end{lemma}
\begin{lemma} (Symmetry) If $f\in\Ext^i_{\mathfrak{a}}(A,B)$, $g\in \Ext^j_{\mathfrak{a}}(B,A)$ then
\[
\Tr_B(fg)=(-1)^{ij}\Tr_B(gf)
\]
in $\HH_\ast(\mathfrak{a})$.
\end{lemma}
From the previous lemma one obtains
\begin{lemma} (Naturality)
If $f\in \Hom_{\mathfrak{a}}(A,B)$ is invertible and $g\in \Ext^\ast_{\mathfrak{a}}(A,A)$ then
\[
\Tr_B(fgf^{-1})=\Tr_A(g)
\]
In other words:
up to the natural identification of $\Ext^\ast_{\mathfrak{a}}(A,A)$
and $\Ext^\ast_{\mathfrak{a}}(B,B)$
the traces $\Tr_A$, $\Tr_B$ coincide.
\end{lemma}
Naturality implies that we may define the trace map on
compact objects $C$ in $D(\mathfrak{a})$. We take an object $\tilde{C}$  in $\tilde{\mathfrak{a}}$
representing $C$ and put $\Tr_C=\Tr_{\tilde{C}}$. By naturality this yields a well
defined map
\[
\Ext^\ast_\mathfrak{a}(C,C)\r \HH_{-\ast}(\mathfrak{a})
\]
Now we come to the additivity of traces. This is a slightly subtle problem. See \cite{Ferrand}.
\begin{lemma} (Additivity)
Assume that we have a commutative diagram
\begin{equation}
\label{ref-A.4-9a}
\xymatrix{
A\ar[r]^u\ar[d]_f &B\ar[d]^g\\
A\ar[r]_u & B
}
\end{equation}
of closed maps in $\tilde{\mathfrak{a}}$ with $|u|=0$. Let $C=\cone u$ and let  $h:C\r C$ be the
morphism obtained from $f,g$ by functoriality of cones in DG-categories. 
Then
\[
\Tr_B(g)=\Tr_A(f)+\Tr_C(h)
\]
\end{lemma}
\begin{proof}
Let $\mathfrak{b}$ be the category consisting of triples $(A,B,u)$
with $A,B\in\tilde{\mathfrak{a}}$ and $u:A\r B$ a closed map of
degree zero.  It is easy to see that $\mathfrak{b}$ has a semi-orthogonal
decomposition $(\tilde{\mathfrak{a}},\tilde{\mathfrak{a}})$ and hence by \cite[Thm 2.4b,c]{kellerexact}
we have 
\begin{equation}
\label{ref-A.2-6}
\HH_\ast(\mathfrak{b})=\HH_\ast(\tilde{\mathfrak{a}})\oplus \HH_\ast(\tilde{\mathfrak{a}})=\HH_\ast(\mathfrak{a})\oplus \HH_\ast(\mathfrak{a})
\end{equation}
There are three canonical functors $\pi_{1,2,3}:\mathfrak{b}\r \tilde{\mathfrak{a}}$ which send $(A,B,u)$ respectively to $A,B$ and $\cone u$. Using \eqref{ref-A.2-6} one easily checks that 
\begin{equation}
\label{ref-A.3-7}
\HH_\ast(\pi_2)=\HH_\ast(\pi_1)+\HH_\ast(\pi_3)
\end{equation}
Now \eqref{ref-A.4-9a} may be viewed as a morphism in $\mathfrak{b}$ which
we denote by $(f,g)$. By functoriality we have
\begin{align*}
\HH_\ast(\pi_1)(\Tr_u(f,g))&=\Tr_A(f)\\
\HH_\ast(\pi_2)(\Tr_u(f,g))&=\Tr_B(g)\\
\HH_\ast(\pi_3)(\Tr_u(f,g))&=\Tr_C(h)
\end{align*}
From \eqref{ref-A.3-7}  we then obtain what
we want.
\end{proof}
We  actually use the following variant of additivity.
\begin{lemma} (Additivity)
\label{ref-A.5-8}
Assume that we have a commutative diagram
\begin{equation}
\label{ref-A.4-9}
\xymatrix{
A\ar[r]^u\ar[d]_f &B\ar[d]^g\\
A\ar[r]_u & B
}
\end{equation}
of closed maps in $\tilde{\mathfrak{a}}$ with $|u|=0$. Assume that
$u$ is injective and $C=\coker u$ is perfect. Let $h:C\r C$ be
obtained by functoriality of cokernels. Then
\[
\Tr_B(g)=\Tr_A(f)+\Tr_C(h)
\]
\end{lemma}
\begin{proof}
This follows easily from the fact that we have a natural commutative diagram
\[
\xymatrix{
&&\cone u\ar[dd]\\
A\ar[r]_u &B\ar[dr]\ar[ru]&\\
&&C
}
\]
where the vertical map is a homotopy equivalence, together with naturality.
\end{proof}
\begin{remark} Additivity holds in fact for  morphisms
of distinguished triangles in Neeman's alternative definition of the derived category. See \cite{Neeman10}.
\end{remark}
\begin{remark} A natural way to formulate additivity is via the filtered derived category. See \cite[V3.7.7, p.310]{Illusie1}.
\end{remark}
\section{Gabber's theorem}
\label{ref-B-10}
The following result is a slight generalization of the main result
of \cite{Kaledin1}.
\begin{proposition}
\label{ref-B.1-11} Consider the following data
\begin{itemize}
\item
$A$ is an equicharacteristic regular local ring with maximal ideal $m$
and residue field $K$ of characteristic $p$.
\item $A'$ is a first order deformation of $A$.
\item $M$ is a finite length $A$-module such that there exists a first order $A'$-defor\-mation $M'$ of $M$.
\item The characteristic of $K$ is either zero or prime to the length of $M$.
\item
$\{-,-\}$ is the induced  bracket on $A$.
\end{itemize}
With these notations we have that $m$ is involutive. I.e.\ $\{m,m\}\subset m$.
\end{proposition}
\begin{remark} This is more or less the setup \cite[\S3]{Kaledin1}. However
there are some differences. 
\begin{itemize}
\item We do not assume that $A'$ is defined over the same ground field as $A$.
\item We do not assume $p=0$.
\item Kaledin does not assume that $(A,m)$ is local. However during the
proof $A$ is localized at $m$. So it is sufficient to consider the local case.
\end{itemize}
\end{remark}
We will now prove Proposition \ref{ref-B.1-11} following the ideas in
\cite[\S3]{Kaledin1}.  However we will not
use Hochschild cohomology of abelian categories. That part of
Kaledin's argument does not generalize to our more general setting but
luckily it is not needed.

\medskip

First of all we may replace $A$ by its completion. Then $A$ contains
a copy of its residue field which we will denote by $K$ as well.
Let  $\Cscr$ be the category of finite
length $A$-modules. 
We may view $\Cscr$ as a $K$-linear category.

If $S\in \Cscr$ then there is an obstruction \cite{lowen4}  $o_S\in \Ext^2_\Cscr(S,S)$ against the existence of a first order deformation
of $S$ over $A'$.
To construct it we choose  projective resolutions $(P_S,d)$ of $S$ over $A$. 
Let $P_S^{\prime}$ be a lifting of $P_S$ to $A'$ (as graded projective $A$-module)
and let $d':P_S'\r P'_S$ be a lifting of $d$ (as an endomorphism of a graded
projective $A$-module). Then $o_S:P_S\r P_S[2]$ is defined
by
$ho_S=(d')^2$ (with obvious notation).

A priori we have $o_S\in \Hom_{D_\Cscr(A)}(S,S[2])$. However it is
well-known that in this case the canonical map $D(\Cscr)\r D_\Cscr(A)$
is an isomorphism so that we may interprete $o_S$ as an element of
$\Ext^2_\Cscr(S,S)$, if we so prefer.

If we have an exact sequence in $\Cscr$
\[
0\r N\r S\r Q\r 0
\]
then the above construction yields a commutative diagram of complexes
\[
\xymatrix{
0\ar[r] &P_N\ar[r]\ar[d]_{o_N} &P_S \ar[r]\ar[d]_{o_S}& P_Q\ar[r]\ar[d]^{o_Q}& 0\\
0\ar[r] &P_N[2]\ar[r] &P_S[2] \ar[r]& P_Q[2]\ar[r]& 0
}
\]
with the horizontal rows being exact. Thus by Lemma \ref{ref-A.5-8}
\[
\Tr_S(o_S)=\Tr_N(o_N)+\Tr_Q(o_Q)
\]
where $\Tr(-)$ is the trace map $\Ext^\ast_{\Cscr}(-,-)\r \HH_{-\ast}(\Cscr)$
(see \S\ref{ref-A-4}). Here in the expression $ \HH_{-\ast}(\Cscr)$ we regard $\Cscr$ as a $K$-linear DG-category via
projective resolutions over $A$.

Applying this to $M$ as in the statement of Proposition \ref{ref-B.1-11} we
find
\[
0=\Tr_M(o_M)=(\length M) \Tr_K(o_K)
\]
($\Tr_M(o_M)=0$ since $M$ has a deformation $M'$).
We conclude $\Tr_K(o_K)=0$. Now $D_{\Cscr}(A)$ has a compact generator
$K$ and hence $D_{\Cscr}(A)\cong D(\RHom_A(K,K))$
 and by Koszul duality, $\RHom_A(K,K)$  is formal and quasi-isomorphic to $A^!=\Ext^\ast_A(K,K)$ which is commutative.
 From this we immediately deduce that
\[
\Tr_K:\Ext^{\ast}_{\Cscr}(K,K)\cong A^!\r \HH_{-\ast}(A^!)\cong \HH_{-\ast}(\Cscr)
\]
is injective. Thus $o_K=0$
 and hence $K$ has a first order deformation.
It is easily seen that this is equivalent to $\{m,m\}\subset m$.

\def\cprime{$'$} \def\cprime{$'$} \def\cprime{$'$}
\providecommand{\bysame}{\leavevmode\hbox to3em{\hrulefill}\thinspace}
\providecommand{\MR}{\relax\ifhmode\unskip\space\fi MR }
\providecommand{\MRhref}[2]{%
  \href{http://www.ams.org/mathscinet-getitem?mr=#1}{#2}
}
\providecommand{\href}[2]{#2}

\end{document}